\documentclass[12pt, a4paper]{article}

\usepackage{a4,amsmath,amssymb, amsthm, latexsym, color, graphicx,url}
\usepackage{tikz}
\usetikzlibrary{arrows.meta}
\usetikzlibrary{decorations.markings}
\usetikzlibrary{calc}

\usepackage{authblk}
\usepackage{fullpage}
\usepackage{enumitem}
\usepackage{xcolor}
\newtheorem{theorem}{Theorem}[section]

\newtheorem{lemma}[theorem]{Lemma}
\newtheorem{corollary}[theorem]{Corollary}

\theoremstyle{definition}
\newtheorem{example}[theorem]{Example}
\newtheorem{definition}[theorem]{Definition}

\newtheorem{construction}[theorem]{Construction}

\title{Beyond uniform cyclotomy}
\author[1]{Sophie Huczynska}
\author[2]{Laura Johnson}
\author[3]{Maura B. Paterson}
\affil[1]{School of Mathematics and Statistics, University of St Andrews, St
Andrews, KY16 9SS, Scotland, UK}
\affil[2]{School of Mathematics, University of Bristol, Bristol, BS8 1UG, UK}
\affil[3]{School of Computing and Mathematical Sciences, Birkbeck, University of London,
Malet St, London, WC1E 7HX, UK}
\date{MSC codes: Primary 11T22, Secondary 05B25, 05B10}                
\begin{document}
\maketitle

\begin{abstract}
Cyclotomy, the study of cyclotomic classes and cyclotomic numbers, is an area of number theory first studied by Gauss. It has natural applications in discrete mathematics and information theory.  Despite this long history, there are significant limitations to what is known explicitly about cyclotomic numbers, which limits the use of cyclotomy in applications. The main explicit tool available is that of uniform cyclotomy, introduced by Baumert, Mills and Ward in 1982. In this paper, we present an extension of uniform cyclotomy which gives a direct method for evaluating all cyclotomic numbers over  ${\rm GF}(q^n)$ of order dividing \makebox{$(q^n-1)/(q-1)$}, for any prime power $q$ and $n \geq 2$, which does not use character theory nor direct calculation in the field. This allows the straightforward evaluation of many cyclotomic numbers for which other methods are unknown or impractical,  extending the currently limited portfolio of tools to work with cyclotomic numbers.  Our methods exploit connections between cyclotomy, Singer difference sets and finite geometry.
\end{abstract}

\section{Introduction}
Cyclotomy, the study of cyclotomic classes and cyclotomic numbers, is an area of number theory first studied by Gauss \cite{Gau}. It is of considerable theoretical interest and has important applications, both to other areas of pure mathematics such as combinatorics and graph theory, and beyond pure mathematics to communications, information security and experiment design.  Many constructions for difference sets, difference families, strongly regular graphs, binary sequences, codes and other combinatorial structures are based on cyclotomy 
\cite{BroWilXia, CarLiMes, ChaCheZho, FanWenFu, FenMomXia, MomWanXia, SchWhi, WanTao, WenYanFuFen, Wil, YanYaoZhe}. Work on evaluating cyclotomic numbers has been carried out for many decades \cite{BauFre, Gau, Leh, LeoWil, Sto, Whi,  XiaYan, Zee}.

Despite this long history however, there are very significant limitations to what is known explicitly about cyclotomic numbers. Expressions for cyclotomic numbers are most commonly given via character sums (typically, Jacobi sums).  However, results are currently not known for orders beyond order 24 \cite{BerEvaWil} and not all known results are given explicitly \cite{FenMomXia}.  Moreover the expressions obtained from character sums are awkward to work with or evaluate explicitly, require the use of tables, and many contain sign ambiguity.  (We note that in  \cite{NguSch}, methods for fast computation of Gauss sums were presented; these may be used to obtain cyclotomic numbers via Jacobi sums.)  It was as recently as 1987 \cite{KatRaj} that an expression without sign ambiguity was obtained for the cyclotomic numbers of order $4$.  Hence, although cyclotomic classes have very useful structural properties which make them desirable to use in various applications, practictioners are generally limited to using specific small values of $e$, via direct calculation in the finite fields.

The cyclotomic classes of order $e$ of ${\rm GF}(q^n)$ ($e \mid q^n-1$), written $C_i^e$ ($0 \leq i \leq e-1$), are the multiplicative cosets of the subgroup of index $e$ in ${\rm GF}(q^n)^*$.  The cyclotomic number $(i,j)_e$ counts the number of solutions to the equation $z_i+1=z_j$ where $z_i \in C_i^e, z_j \in C_j^e$.

To date, the main theoretical approach that offers a direct approach to the evaluation of cyclotomic numbers is uniform cyclotomy.  Uniform cyclotomy (sometimes called the semi-primitive case) has been used to construct a variety of structures including strongly regular graphs \cite{BroWilXia, FenMomXia, MomWanXia}, cyclic codes \cite{YanYaoZhe}, constacyclic codes \cite{FanWenFu}, linear codes with small hull \cite{CarLiMes, WanTao}, external difference families \cite{WenYanFuFen}, partial geometric designs \cite{ChaCheZho} and many more.  The following definition was introduced by Baumert, Mills and Ward in \cite{BauMilWar}.

\begin{definition}\label{Baumert}
The cyclotomic numbers $(i,j)_e$ over $ {\rm GF}(q)$ are \emph{uniform} if $(0,i)_e=(i,0)_e=(i,i)_e=(0,1)_e$ for $i \neq 0$, and $(i,j)_e=(1,2)_e$ for $0 \neq i \neq j \neq 0$.
\end{definition}

The cyclotomic numbers of order $2$ are known explicitly \cite{Gau, Sto}.  They are uniform precisely when $f$ is even, in which case $(0,0)_2=(f-2)/2$ and $(0,1)_2=(1,0)_2=(1,1)_2=f/2$.  The following result is proved in \cite{BauMilWar}:

\begin{theorem}\label{BMWTheorem}
Let $q$ be a power of a prime $p$ and let $q=ef+1$, $e \geq 3$.  The cyclotomic numbers of order $e$ over ${\rm GF}(q)$ are uniform if and only if $-1$ is a power of $p$ modulo $e$. \\
 If this holds, then either $p=2$ or $f$ is even; $q = r^2$ with $r \equiv 1 \mod e$; and setting $\eta=\frac{r-1}{e}$, we have:
\begin{itemize}
    \item[(i)] $(0,0)_e = \eta^2-(e-3)\eta-1$
    \item[(ii)] $(0,i)_e = (i,0)_e = (i,i)_e = \eta^2+\eta$ for $i\neq{0}$
    \item[(iii)] $(i,j)_e =\eta^2$ for  $0 \neq i \neq j$.
    \end{itemize}
\end{theorem}

Thus, if the cyclotomic numbers of order $e$ are uniform over $ {\rm GF}(q)$, then $q$ is a square; say $q=p^{2s}$ for some $s \in \mathbb{N}$.  It is shown in \cite{BauMilWar} that the cyclotomic numbers of order $e \mid p^s+1$ are uniform, by first establishing the following:

\begin{theorem}\label{BMWthm1}
Consider $ {\rm GF}(p^{2s})$ and let $p^{2s}=ef+1$. If $e=p^s+1$ then:
\begin{itemize}
\item[(i)] $(0,0)_e=f-1$
\item[(ii)] $(0,i)_e=(i,0)_e=(i,i)_e=0$ for all $0<i<e$
\item[(iii)] $(i,j)_e=1$ for $0 < i\neq j \leq e$.
\end{itemize}
\end{theorem}

What underpins this uniform cyclotomy result, is the guaranteed presence of a subfield.  Consider the proof of Theorem \ref{BMWthm1} in \cite{BauMilWar}.  Let $\alpha$ be a primitive element of $ {\rm GF}(q)$; then $\alpha^e$ is a generator of $ {\rm GF}(p^s)$, i.e. the cyclotomic class $C_0^e=\langle \alpha^e \rangle$ is $ {\rm GF}(p^s)^*$ (see Definition~\ref{def:cclass}).  Closure under subtraction guarantees that $\alpha^{es+i}+1=\alpha^{et}$ is impossible for $0 <i<e$, and hence $(i, 0)_e=(0,i)_e=(i,i)_e=0$ for all non-zero $i$.  For $i \neq j$ both nonzero, consideration of the equation $\alpha^{es+i}+1=\alpha^{et+j}$ leads to forming the quadratic extension field of $ {\rm GF}(p^s)$ by adjoining $\alpha^j$; the uniqueness of the representation for $\alpha^i$ in this field guarantees $(i,j)_e=1$ for all $0< i \neq j <e$.

We present an alternative characterisation of the uniformity condition in Theorem \ref{BMWTheorem} which motivates this paper:
\begin{theorem}\label{alt_char}
Let $p$ be a prime, $s \in \mathbb{N}$ and let $p^{2s}=ef+1$ for some $e,f \in \mathbb{N}$, $e \geq 3$.
\begin{itemize}
\item[(i)] $-1 \equiv p^x \mod e$ for some $x \in \mathbb{N}$ if and only if $-1 \equiv p^t \mod e$ for some $t \mid s$, i.e. if and only if $e \mid p^t+1$ for some $t \mid s$.
\item[(ii)] Let $t \mid s$.  If $e \mid p^t+1$, then $e \mid (p^{2s}-1)/(p^d-1)$ for any $d \mid t$.
\end{itemize}
\end{theorem}
\begin{proof}
(i) Suppose that  $-1 \equiv p^x \mod e$ for some $x \in \mathbb{N}$.  Let $r$ be the smallest positive integer such that  $-1 \equiv p^r \mod e$.   Use the Division Algorithm to write $s=ar+b$ for some $a,b \in \mathbb{Z}$ with $0 \leq b <r$.  Now,  $p^{2s} \equiv 1 \mod e$, so $p^s \equiv \pm 1 \mod e$.  Then
$$ \pm 1 \equiv p^s \equiv p^{ar+b}  \equiv (p^r)^a \cdot p^b \equiv \pm p^b \mod e.$$
So $p^b \equiv \pm 1 \mod e$.  We claim that $b=0$.  Suppose $b>0$.  By minimality of $r$, $p^b \not\equiv -1 \mod e$.  So $p^b \equiv 1 \mod e$, and hence the sequence $p^0, p^1, p^2, \ldots$ taken modulo $e$ contains only the elements of $\{1,p,p^2, \ldots, p^{b-1}\}$ taken modulo $e$, repeated with period $b$.  By minimality of $r$,  $p^i \not \equiv -1 \mod e$ for all $0 \leq i \leq b-1$ and so  $p^i \not \equiv -1 \mod e$ for all $i \geq 0$, a contradiction.  So $b=0$ and $r \mid s$.  Hence we can take $t=r$.  The reverse direction is immediate.\\
(ii) Let $t \mid s$ and suppose $d \mid t$.  Then we have the integer factorisation:
$$ \frac{p^{2s}-1}{p^d-1}=(p^t+1)\left(\frac{p^t-1}{p^d-1}\right)\left( \frac{p^{2s}-1}{p^{2t}-1}\right).$$
\end{proof}

In this paper, we present results which extend uniform cyclotomy.  We give a combinatorial characterisation and direct evaluation method for all cyclotomic numbers of order $e (\geq 2)$ where $e \mid (q^n-1)/(q-1)$, for any $n \geq 2$ and any prime power $q$.   By Theorem \ref{alt_char},  taking $q=p^s$ and $n=2$ yields the uniform cyclotomy result of \cite{BauMilWar} with $e \mid p^s+1$, while for the other uniform cyclotomy cases where $e \mid p^t+1$ and $t$ is a proper divisor of $s$, we take $q=p^d$ and $n=2s/d$ for suitable $d \mid t$.   While the expressions in the $n>2$ cases are not quite as simple as in the $n=2$ case, we provide an explicit theoretical description allowing straightforward direct evaluation working in  $\mathbb{Z}_e$ ($e \in \mathbb{N}$), which does not involve character sums nor computation in the finite field.  This allows evaluation of many cyclotomic numbers for which no direct methods are known. 

A key computational contribution of the paper is to extend the currently limited portfolio of tools available to work with cyclotomic numbers, and to provide greater breadth without increasing computational requirements. Depending on the parameters used, and methods of storage/lookup, the computation complexity of our approach is comparable to the naive method of direct evaluation in the finite field, though likely to be better in some small cases.  However, in practice, practitioners do not typically perform the general computation outside of a very narrow window.    Hence one important contribution of our approach is to give a practical and accessible method for working directly with a wider range of cyclotomic numbers than is currently considered.  Our approach facilitates human/computer interaction and can be applied at different levels of computer involvement (from using the computer to produce the initial Singer difference set, to using it to perform the whole calculation).  The algorithm is straightforward to implement and does not require calculations in extension fields or any other techniques that would require a computer algebra package.

\section{Background}\label{sec:background}

\subsection{Cyclotomy}
We follow the definitions of \cite{Sto}.  Throughout, let $q$ be a power of a prime $p$ and let $ {\rm GF}(q)$ denote the (unique up to isomorphism) finite field of order $q$.  Let $\alpha$ be a primitive element of $ {\rm GF}(q)$.  

\begin{definition}\label{def:cclass}
Let $q=ef+1$ where $e,f \in \mathbb{N}$.  The cyclotomic classes $C_i^e$ in $ {\rm GF}(q)$ of order $e$ ($0 \leq i \leq e-1$) are defined to be:
$$ C_i^e=\{\alpha^{es+i}: 0 \leq s \leq f-1\}.$$
Note that $C_0^e$ is the multiplicative subgroup of $ {\rm GF}(q)$ of cardinality $f$.
\end{definition}

\begin{definition}
Let $q=ef+1$ where $e,f \in \mathbb{N}$.  The cyclotomic numbers $(i,j)_e$ with $0 \leq i,j \leq e-1$ are defined to be the number of solutions to the equation 
$$ z_i=z_j-1$$
where $z_i \in C^e_i$ and $z_j \in C^e_j$.  This is the number of ordered pairs $(s,t)$ with $0 \leq s,t \leq f-1$ such that 
$$ \alpha^{es+i}=\alpha^{et+j}-1$$
\end{definition}

The following properties are known for cyclotomic numbers (see \cite{Sto, WenYanFuFen}):
\begin{lemma}\label{Storer}
\begin{itemize}
\item[(i)] For any $m,n \in \mathbb{Z}$, $(i,j)_e=(i+me, j+ne)_e$
\item[(ii)] $(i,j)_e=(e-i,j-i)_e$
\item[(iii)] $(i,j)_e$ equals $(j,i)_e$ if $p=2$ or if  $p$ is odd and $f$ is even, and $(i+\frac{e}{2}, j+\frac{e}{2})_e$ if $p$ is odd and $f$ is odd
\item[(iv)] $(i,j)_e=(pi,pj)_e$
\item[(v)] $\sum_{i=0}^{e-1} (i,j)_e =f-\delta_{0j} $ where $\delta_{0j}=1$ if $j=0$ and $0$ if $j \neq 0$.
\end{itemize}
\end{lemma}

Character sum expressions are known for many cyclotomic numbers of order up to $e=24$ \cite{BerEvaWil}.  Some upper bounds on cyclotomic numbers are also known \cite{BetHirKomMun} and a congruence expression is given in \cite{Tha}.  We will provide a method to directly evaluate cyclotomic numbers by exploiting connections with linear algebra and finite geometry.  

\subsection{The vector space viewpoint}
Let $q$ be a power of some prime $p$ and $n\in \mathbb{N}$ with $n\geq 2$.  We may consider $ {\rm GF}(q^n)$ as a dimension-$n$ vector space over $ {\rm GF}(q)$ with basis $\{1,\alpha,\alpha^2, \ldots, \alpha^{n-1} \}$, where $\alpha$ is a primitive element of $ {\rm GF}(q^n)$.  So \makebox{$ {\rm GF}(q^n)=\{a_0 + a_1 \alpha + \cdots + a_{n-1} \alpha^{n-1}: a_i\in {\rm GF}(q), 0 \leq i \leq n-1  \}$}  and in particular each element $\alpha^k$ of $ {\rm GF}(q)^*$ ($0 \leq k \leq q-1$) has a unique expression as a nonzero polynomial in $\alpha$ over $ {\rm GF}(q)$, of degree at most $n-1$.  The polynomial $a_0+ a_1 \alpha+ \cdots a_{n-1}\alpha^{n-1}$ may be expressed as the $n$-tuple $(a_0,\ldots,a_{n-1})$. We will work in the projective space ${\rm PG}(n-1,q)$, i.e. we will take the non-zero $n$-tuples and identify two $n$-tuples if they differ by a scalar factor.  This is equivalent to considering the factor ring $ {\rm GF}(q^n)^* / {\rm GF}(q)^*$.  

Viewing $ {\rm GF}(q^n)$ as a vector space allows us to work with its vector subspaces.  Of particular interest are the subspaces of codimension $1$ (hence dimension $n-1$). When we take the quotient of one of these spaces by the scalars, we obtain a hyperplane of ${\rm PG}(n-1,q)$.  A classic combinatorial application of these hyperplanes is in constructing Singer difference sets (see \cite{Ber, Sin, Sto}).  A hyperplane may be defined by a suitable linear equation in the coordinates. For example, we could take $a_0=0$, or $a_{n-1}=0$, or $\operatorname{Tr}_{n/1}(a_0 +a_1 \alpha + \cdots + a_{n-1}\alpha^{n-1})=0$ \cite{Handbook}.  Multiplying elements of ${\rm GF}(q^n)$ by $\alpha$ is a ${\rm GF}(q)$-linear operation that maps hyperplanes to hyperplanes, and in fact permutes all the  hyperplanes of ${\rm PG}(n-1,q)$ in a single cycle.

\begin{definition}
Let $G$ be a group of order $n$, written additively. A $k$-subset $D$ of $G$ forms an $(n,k,\lambda)$-Difference Set (DS) if the multiset \makebox{$\{x-y: x,y \in D, x \neq y\}$} comprises $\lambda$ copies of each non-identity element of $G$.
\end{definition}

We have the following well-known result of Singer \cite{Sin}:
\begin{theorem}[Singer]
Modulo the scalars, the elements of a hyperplane form a multiplicative $\left(\frac{q^n-1}{q-1}, \frac{q^{n-1}-1}{q-1}, \frac{q^{n-2}-1}{q-1}\right)$-difference set in the cyclic group of order $\frac{q^n-1}{q-1}$ induced by $\alpha$ on the equivalence classes. 
\end{theorem} 

Multiplication by $\alpha$ induces a single cycle on the points, as well as on the hyperplanes.  Taking each power $i$ such that $\alpha^i$ is in this set yields a corresponding difference set in the additive group $\left(\mathbb{Z}_{\frac{q^n-1}{q-1}},+\right)$. This is known as a Singer difference set \cite{Handbook}. It is said to be in standard form if it contains $0$ and $1$.  If we construct \makebox{$\operatorname{Dev}(D)=\{i+D: 0 \leq i \leq \frac{q^n-1}{q-1}-1\}$} in $\left(\mathbb{Z}_{\frac{q^n-1}{q-1}},+\right)$, we obtain a collection of Singer difference sets with the same parameters \cite{Ber}.  
Addition of an integer modulo $(q^n-1)/(q-1)$ cyclically permutes these difference sets in a single cycle.  Further, $\operatorname{Dev}(D)$ forms a symmetric balanced incomplete block design (BIBD) with parameters  $\left(\frac{q^n-1}{q-1}, \frac{q^{n-1}-1}{q-1}, \frac{q^{n-2}-1}{q-1}\right)$.

Note that a different collection of equivalent difference sets may be obtained via a different choice of primitive element $\alpha$, but that all such difference sets are related (for any two such Singer DSs $D_1$ and $D_2$, there is some integer $t$ such that $tD_1$ is equivalent to $D_2$ \cite{Ber}).

The following straightforward process to calculate a Singer difference set from a given primitive polynomial is given in VI.18 of \cite{Handbook}. 

\begin{construction}\label{Handbook}
Let $f(x)=x^n+\sum_{i=1}^n a_i x^{n-i}$ be a primitive polynomial of degree $n$ over $ {\rm GF}(q)$.  Define the recurrence relation $\gamma_r=-\sum_{i=1}^n a_i \gamma_{r-i}$.  Take arbitrary initial values; using $\gamma_0=0,\ \gamma_1=0$ will yield a difference set in standard form.  Then the set of integers $\{0 \leq i \leq \frac{q^n-1}{q-1}: \gamma_i=0\}$ is a Singer difference set.
\end{construction}

Let $e=(q^n-1)/(q-1)$ and $f=q-1$; then $q^n=ef+1$. Since $ {\rm GF}(q)^*=C_0^e$, the elements of the factor ring $ {\rm GF}(q^n)^*/{\rm GF}(q)^*$ correspond to the set of $e$ projective points \begin{equation*}P=\{C_0^e, \alpha C_0^e, \alpha^2 C_0^e, \ldots, \alpha^{e-1}C_0^e\}=\{C_i^e: 0 \leq i \leq e-1\}.\end{equation*} A hyperplane of ${\rm PG}(n-1,q)$ consists of $(q^{n-1}-1)/(q-1)$ points, hence there are $(q^{n-1}-1)/(q-1)$ elements $\alpha^i C_0^e$ of $P$ that lie in a given hyperplane.  The set of values $i$ corresponding to the classes $C_i^e$ whose union forms the hyperplane yield the $\left(e, \frac{q^{n-1}-1}{q-1}, \frac{q^{n-2}-1}{q-1}\right)$ Singer difference set $I$ in $(\mathbb{Z}_e,+)$. 

The following structure will play an important role in what follows; for more details, see \cite{Handbook}.
\begin{definition}
A finite projective plane (FPP) is a finite set of points, a finite set of lines and an incidence relation between them, such that:
\begin{itemize}
\item[(i)] any two distinct points are incident with exactly one line
\item[(ii)] any two distinct lines are incident with exactly one point
\item[(iii)] there exists a quadrangle (four distinct points of which no three are collinear).
\end{itemize}
\end{definition}

When $n=3$, a Singer difference set $D$ has parameters $(q^2+q+1,q+1,1)$, and the development of a Singer difference set  forms a finite projective plane $\Pi_D$ of order $q$ whose points are the elements of $\mathbb{Z}_{q^2+q+1}$ and whose lines are given by \makebox{$\{i+D : 0 \leq i \leq q^2+q\}$}.

\section{Determining cyclotomic structure and cyclotomic numbers}

A major application of cyclotomic classes and cyclotomic numbers is in the construction of difference structures. This approach was introduced for difference families by Wilson \cite{Wil} and for difference sets by Storer \cite{Sto}; it has since been considerably extended.  We are therefore interested both in evaluation of the cyclotomic numbers and characterisation of the difference structure of cyclotomic classes; these can be fruitfully related.

For two subsets $A,B$ of a group $G$ written additively, we define the following multisets: $\Delta(A)=\{x-y: x \neq y \in A\}$ and $\Delta(A,B)=\{x-y: x \in A, y \in B\}$.  A useful link is the following result (see \cite{HucJoh}):
\begin{lemma}\label{DM}
\begin{itemize}
\item[(i)] $\Delta(C_0^e)=\cup_{s=1}^{f-1} (\alpha^{es}-1)C_0^e=\cup_{i=0}^{e-1} (i,0)_e C_i^e$
\item[(ii)] For $0<j<e$, $\Delta(C_j^e) = \cup_{i=0}^{e-1}(i,0)_e C_{i+j}^e$
\item[(iii)] For $0<j<e$, $\Delta(C_j^e, C_0^e)=\cup_{s=0}^{f-1} (\alpha^{es+j}-1) C_0^e = \cup_{i=0}^{e-1}(i,j)_e C_i^e$
\item[(iv)] For $0<j, \ell <e$, $\Delta(C_{j+\ell}^e, C_{\ell}^e)= \cup_{i=0}^{e-1}(i,j)_e C_{i+\ell}^e$
\end{itemize}
\end{lemma}
These are generally used to determine which classes appear in the multiset of differences.  Instead we will use difference set structure to determine the cyclotomic numbers.

\subsection{Cyclotomic numbers of order $e$ where $e=(q^n-1)/(q-1)$}

We first consider the fundamental case on which we will build our subsequent results.

Let $e=(q^n-1)/(q-1)$. We observe that for each $0 \leq k \leq e-1$, all elements that occur in $\Delta(C_k^e,C_0^e)=\{ \alpha^{es+k}-\alpha^{et}: 0 \leq s,t \leq f-1 \}$ are contained in the subspace $\operatorname{Span}(1,\alpha^k)$ of the vector space $ {\rm GF}(q^n)$.  In ${\rm PG}(n-1,q)$ this corresponds to the unique line through the points $1$ and $\alpha^k$.  Since  $\operatorname{Span}(1,\alpha^k)$  is a two-dimensional vector space over $ {\rm GF}(q)$, it has $q^2-1$ non-zero elements and upon factoring-out by $ {\rm GF}(q)^*$ we obtain a set comprising the $q+1$ points of the line.

\begin{definition}\label{S_k}
Consider $ {\rm GF}(q^n)$ as a vector space of dimension $n$ at least $2$ over $ {\rm GF}(q)$, let $e=(q^n-1)/(q-1)$ and let $\alpha$ be a primitive element of $ {\rm GF}(q^n)$.  For $1 \leq k \leq e-1$, define
$$S_k=\{j: \alpha^j \in \operatorname{Span}(1,\alpha^k), 0 \leq j \leq e-1\} \subseteq (\mathbb{Z}_e,+).$$
\end{definition}

Note that $0 \in S_k$ for all $1 \leq k \leq e-1$. In the $n=2$ case, $\operatorname{Span}(1,\alpha^k)={\rm GF}(q^2)$ for any $1 \leq k \leq e-1$, so $S_k=\mathbb{Z}_{q+1}$. In the $n=3$ case, we observe that ${\rm PG}(2,q)$ is a projective plane, so the line $S_k$ is in fact a hyperplane of that plane and so $S_k$ is a Singer difference set.  For $n\geq 4$, the line $S_k$ lies in any hyperplane that contains the points $0$ and $k$, and hence will be properly contained in the corresponding Singer difference sets.  In Section~\ref{sec:generaln} we will discuss how $S_k$ may be obtained from these by taking appropriate intersections.

We now present a structural theorem which will be key in what follows.

\begin{theorem}\label{thm:main}
Consider the finite field $ {\rm GF}(q^n)$ and let $e=(q^n-1)/(q-1)$.  View $ {\rm GF}(q^n)$ as the vector space $ {\rm GF}(q)^n$ over $ {\rm GF}(q)$. Let $\alpha$ be a primitive element of $ {\rm GF}(q^n)$.
Then:
\begin{itemize}
\item[(i)] $\Delta(C_k^e, C_0^e)$ contains only classes $C_j^e$ where $j \in S_k$.
\item[(ii)] $C_k^e$ and $C_0^e$ do not occur in $\Delta(C_k^e, C_0^e)$.
\item[(iii)] The other classes $C_j^e$ with $j \in S_k$, $j \not\in \{0,k\}$, occur precisely once each in $\Delta(C_k^e, C_0^e)$, i.e.
$$\Delta(C_k^e,C_0^e)=\cup_{j \in S_k \setminus \{0,k\}} C_j^e.$$
\item[(iv)] For $0 \leq i \neq j \leq e-1$, 
$$\Delta(C_i^e,C_j^e)=\cup_{h \in S_{i,j} \setminus \{i,j\}} C_h^e$$
where $S_{i,j}=\{h: \alpha^h \in \operatorname{Span}(\alpha^i,\alpha^j), 0 \leq h \leq e-1\}=j+S_{i-j}$.
\end{itemize}
\end{theorem}
\begin{proof}
(i) By Lemma \ref{DM}, $\Delta(C_k^e,C_0^e)$ is the multiset union of the $f$ cyclotomic classes $(\alpha^{es+k}-1)C_0^e$ ($0 \leq s \leq f-1$). Since $C_0^e=\langle \alpha^e \rangle= {\rm GF}(q)^*$, the cyclotomic class $C_k^e=\{\alpha^{es+k}: 0 \leq s \leq f-1\}=\{ \alpha^{es} \alpha^k: 0 \leq s \leq f-1\}=\alpha^k  {\rm GF}(q)^*$.  So all elements $\alpha^{es+k}-1$ ($0 \leq s \leq f-1$) are of the form $a \alpha^k+b$ for some $a,b \in  {\rm GF}(q)^*$, i.e. lie in $\operatorname{Span}(1,\alpha^k)$.  The same holds for all elements of $(\alpha^{es+k}-1)C_0^e$, hence all classes in $\Delta(C_k^e,C_0^e)$ are from the set of classes $C_j^e$ where $j \in S_k$.\\
(ii) Let $x \in C_k^e$ where $k \neq 0$.  If $x-1=y \in C_k^e$ then $x-y=1$.   Since $ \Delta(C_k^e)=\alpha^k \Delta(C_0^e)$ contains only elements of $C_k^e$, this implies $1= x-y \in \Delta(C_k^e)$, a contradiction.  If $x-1=a \in C_0^e$ then $x=a+1 \in C_0^e$, a contradiction.\\
(iii) Let $x \neq y \in C_k^e$ and suppose $x-1, y -1 \in C_j^e$ for some $1 \leq j \neq k \leq e-1$.  Then $(x-1)-(y-1)=x-y \in C_j^e$; however $x-y \in C_k^e$, a contradiction since $j \neq k$.  Hence no class occurs more than once.  Observe that $\operatorname{Span}(1, \alpha^k)$ has $q^2-1$ nonzero elements; factoring out by $ {\rm GF}(q)^*$ there are $q+1$ cyclotomic classes forming the subspace, i.e. $|S_k|=q+1$.  From above, $C_0^e$ and $C_k^e$ do not occur in $\Delta(C_k^e,C_0^e)$; therefore the $q-1=f$ remaining classes must occur without repetition.\\
(iv) This follows from (iii) since $\Delta(C_i^e,C_j^e)=\alpha^j \Delta(C_{i-j}, C_0)$.
\end{proof}

This yields the following characterization of the cyclotomic numbers in this setting:
\begin{theorem}\label{cor:main}
Consider the finite field $ {\rm GF}(q^n)$ and let $e=(q^n-1)/(q-1)$. Let $\alpha$ be a primitive element of $ {\rm GF}(q^n)$. Then the cyclotomic numbers of order $e$ are as follows:
\begin{itemize}
\item[(i)] $(0,0)_e=f-1$
\item[(ii)] $(0,i)_e=(i,0)_e=(i,i)_e=0$ for all $1 \leq i \leq e-1$
\item[(iii)] For $1 \leq j \leq e-1$ and $i \not\in \{0,j\}$, $(i,j)_e$ equals $1$ if $i \in S_j \setminus \{0,j\}$ and $0$ otherwise
\item[(iv)] For $1 \leq i,j \leq e-1$ with $i \neq j$, 
$$(i,j)_e =\begin{cases}
    1, \mbox{ if $\{i,j\}$ occur together in some $S_h$ ($0<h<e$)}\\
    0, \mbox{ otherwise}
\end{cases}$$
\end{itemize}
\end{theorem}
\begin{proof}
Part (i) holds since $C_0^e \cup \{0\}$ is a subfield.  Part(ii) holds by Lemma \ref{Storer} since $\sum_{i=0}^{e-1} (i,0)_e=f-1$ and $(i,0)_e=(0,i)_e=(e-i,e-i)_e$.  For (iii), it is known that $\Delta(C_j^e,C_0^e)=\cup_{i=0}^{e-1} (i,j)_e C_i^e$. Theorem \ref{thm:main} shows that for $0 <j <e$, each $C_i^e$ with $i \in S_j \setminus \{0,j\}$ occurs precisely once in $\Delta(C_j^e,C_0^e)$, and no other classes occur.  For (iv), let $h \in \{1,\ldots, e-1\}$.   For any nonzero $j \in S_h$, we obtain precisely $q-1$ nonzero cyclotomic numbers $(i,j)_e=1$, where each $i \in S_h \setminus \{0,j\}$. Running through each of the $q$ nonzero values of $j \in S_h$, we obtain $q(q-1)$ cyclotomic numbers, corresponding to all possible ordered pairs of distinct non-zero elements from $S_h$.
\end{proof}  

When $n=2$, we immediately recover the uniform cyclotomy situation.
\begin{corollary}
Consider the finite field $ {\rm GF}(q^2)$; let $e=(q^2-1)/(q-1)=q+1$ and $f=q-1$. Let $\alpha$ be a primitive element of $ {\rm GF}(q^2)$.
Then for $1 \leq k \leq e-1$, 
$$\Delta(C_k^e,C_0^e)=\bigcup_{j \in \mathbb{Z}_{q+1} \setminus \{0,k\}} C_j^e$$. 
The cyclotomic numbers are as follows:
\begin{itemize}
\item[(i)] $(0,0)_e=f-1$
\item[(ii)] $(0,i)_e=(i,0)_e=(i,i)_e=0$ for all $1 \leq i \leq e-1$
\item[(iii)] For $1 \leq i,j \leq e-1$ with $i \neq j$, $(i,j)_e=1$.
\end{itemize}
\end{corollary}
\begin{proof}
Here, for any $1 \leq k \leq e-1$, $\operatorname{Span}(1,\alpha^k)$ is a subspace of dimension $2$, hence is the whole of $ {\rm GF}(q^2)$, and hence $S_k=\mathbb{Z}_{q+1}=\{0,1,\ldots,q\}=\{0,1,\ldots,e-1\}$.  Every pair of distinct nonzero elements $\{i,j\}$ with $1 \leq i \neq j \leq e-1$ occurs together in this set.
\end{proof}

Observe that, for $n>2$, the property that all $(i,j)_e$ are equal for $1 \leq i \neq j \leq e-1$ no longer holds, although there are now just two possible values ($0$ or $1$) for such $(i,j)_e$.

\subsection{Cyclotomic numbers of order $e=(q^n-1)/(q-1)$ where $n=3$ and $n=4$}
When $n=3$ and $n=4$,  the set $S_k$ ($1 \leq k \leq e-1$) has a particularly nice form.

\begin{theorem}\label{n=3theorem}
Consider the finite field $ {\rm GF}(q^3)$; let $e=(q^3-1)/(q-1)$ and $f=q-1$. Let $\alpha$ be a primitive element of $ {\rm GF}(q^3)$ and let $I$ be a corresponding Singer difference set in $\mathbb{Z}_e$.
Then for $1 \leq k \leq e-1$, 
$$\Delta(C_k^e,C_0^e)=\bigcup_{j \in S_k \setminus \{0,k\}} C_j^e$$ where $S_k$ is the unique Singer difference set in the finite projective plane $\pi_I=Dev(I)$ containing $0$ and $k$.
\end{theorem}
\begin{proof} Only the claim about $S_k$ requires justification.  Since $ {\rm GF}(q^3)$  is a vector space of dimension $3$ over $ {\rm GF}(q)$, the two-dimensional subspace $\operatorname{Span}(1,\alpha^k)$ has codimension $1$ and hence is a hyperplane.  Its corresponding set of powers $S_k$ is a Singer difference set by construction, and its development is a finite projective plane. Since $S_k$ is the line of the FPP containing the pair $0$ and $k$, it is uniquely determined.
\end{proof}

\begin{corollary}
Consider the finite field $ {\rm GF}(q^3)$. Let $e=(q^3-1)/(q-1)$ and $f=q-1$. Let $\alpha$ be a primitive element of $ {\rm GF}(q^3)$; let $I$ be a corresponding Singer difference set in $\mathbb{Z}_e$ and let $S_j$ be the unique line in the FPP $\operatorname{Dev}(I)=\pi_I$ containing $\{0,j\}$.  Then
\begin{itemize}
\item[(i)] $(0,0)_e=f-1$
\item[(ii)] $(0,i)_e=(i,0)_e=(i,i)_e=0$ for all $1 \leq i \leq e-1$
\item[(iii)] For $1 \leq j \leq e-1$ and $i \not\in \{0, j\}$, $(i,j)_e$ equals $1$ if $i \in S_j \setminus \{0,j\}$, and equals $0$ otherwise
\item[(iv)] For $1 \leq i,j \leq e-1$ with $i \neq j$, 
$$(i,j)_e = \begin{cases}
    1, \mbox{ if $\{i,j\}$ occur together in some $S_k$ ($0<k<e$)}\\
    0, \mbox{ otherwise}
\end{cases}$$
\end{itemize}
\end{corollary}

In Section~\ref{sec:background}, a straightforward algorithm was given for producing the Singer difference set via a recurrence relation using modular arithmetic, starting from a given primitive polynomial. For any $j$ with $1\leq j\leq q-1$ there is a single pair of elements $x$ and $y$ in the difference set with $y-x=j \mod e$.  Translating each element of the difference set by $-x$ results in a difference set containing $0$ and $j$, i.e. this translated difference set is none other than $S_j$.  This can be used to obtain the corresponding cyclotomic numbers $(i,j)_e$ without any use of character sums or direct computation in the extension field.

\begin{example}
Consider ${\rm GF}(27)$; we will obtain the cyclotomic numbers of order $e=13$.  For the general case of cyclotomic numbers of order $13$, we are aware only of the result of \cite{Zee} which performs some analysis of the corresponding Jacobi sums via a quaternary quadratic decomposition, and no direct method for evaluation. 

Let $q=3=n$ and let $e=13$ and $f=2$.  Take the primitive polynomial $x^3+2x^2+1$.  The recurrence relation is $\gamma_r=\gamma_{r-1}+2\gamma_{r-3}$, which yields the sequence $0011102112101 \ldots$. Then $I=\{0,1,5,11\}$ is the corresponding Singer difference set in standard form, and the other elements of $\operatorname{Dev}(I)$ containing $0$ are the translates of $I$ by its negatives, i.e. $-1+I=\{12,0,4,10\},-5+I=\{8,9,0,6\}$ and $-11+I=\{2,3,7,0\}$.   From above,
$$\Delta(C_1^{13},C_0^{13})=C_5^{13} \cup C_{11}^{13}, \Delta(C_5^{13},C_0^{13})= C_1^{13} \cup C_{11}^{13}, 
\Delta(C_{11}^{13},C_0^{13}) = C_1^{13} \cup C_5^{13},$$
$$\Delta(C_2^{13},C_0^{13}) = C_3^{13} \cup C_{7}^{13},
\Delta(C_3^{13},C_0^{13})= C_2^{13} \cup C_{7}^{13},\Delta(C_7^{13},C_0^{13}) = C_2^{13} \cup C_3^{13},$$
$$\Delta(C_4^{13},C_0^{13}) = C_{10}^{13} \cup C_{12}^{13},
\Delta(C_{10}^{13},C_0^{13}) = C_4^{13} \cup C_{12}^{13},
\Delta(C_{12}^{13},C_0^{13}) = C_4^{13} \cup C_{10}^{13},$$
$$
\Delta(C_6^{13},C_0^{13}) = C_8^{13} \cup C_{9}^{13},
\Delta(C_8^{13},C_0^{13}) = C_6^{13} \cup C_{9}^{13},
\Delta(C_{9}^{13},C_0^{13}) = C_6^{13} \cup C_8^{13}.$$  
Hence $(0,0)_{13}=1$, $(5,1)_{13}=(11,1)_{13}=(3,2)_{13}=(7,2)_{13}=(2,3)_{13}=(7,3)_{13}=(10,4)_{13}=(12,4)_{13}=(1,5)_{13}=(11,5)_{13}=(8,6)_{13}=(9,6)_{13}=(2,7)_{13}=(3,7)_{13}=(6,8)_{13}=(9,8)_{13}=(6,9)_{13}=(8,9)_{13}=(4,10)_{13}=(12,10)_{13}=(1,11)_{13}=(5,11)_{13}=(4,12)_{13}=(10,12)_{13}=1$ and all other $(i,j)_{13}=0$.
\end{example}

\begin{example}
Consider ${\rm GF}(64)= {\rm GF}(4^3)$; we will obtain the cyclotomic numbers of order $21$. We are aware of no direct method for $e=21$ in general.  Here \makebox{$e=(4^3-1)/(4-1)=21$,} we have $f=3$ and the Singer difference set has size $(4^2-1)/(4-1)=5$. Take primitive polynomial $x^3 + x^2 + x + \gamma^2\in{\rm GF}(4)[x]$, where $\gamma$ is a primitive element of $ {\rm GF}(4)$ (i.e. $\gamma^2=\gamma+1$).  The sequence from the recurrence relation is 
$$ 0, 0, 1, 1, 0, \gamma, 1, \gamma^2, \gamma^2, \gamma^2, \gamma, \gamma^2, \gamma^2, 1, 0, \gamma^2, 0, \gamma^2, 1, \gamma, 1 $$
so the Singer difference set is $\{ 0, 1, 4, 14, 16\}$.  The other $0$-containing translates of $I$ are $-1+I=\{0,3,13,15,20\}, -4+I=\{0,10,12,17,18\}, -14+I=\{0,2,7,8,11\}$ and $-16+I=\{0,5,6,9,19\}$.  Here $(0,0)_{21}=f-1=2$ and all $(0,i)_{21}=(i,0)_{21}=(i,i)_{21}$ for $0<i<21$.  Using $S_1=\{0,1,4,14,16\}$, we have that $\Delta(C_1^{21}, C_0^{21})=C_4^{21} \cup C_{14}^{21} \cup C_{16}^{21}$ and $(4,1)_{21}=(14,1)_{21}=(16,1)_{21}$, while all other $(x,1)_{21}=0$ ($0<x<21$).  The other $(i,j)_{21}$ ($0<i \neq j<21$) can be calculated in a similar way from the listed sets.
\end{example}

When $n=4$, we require an additional step, but still obtain a straightforward process.

\begin{theorem}
Consider the finite field $ {\rm GF}(q^4)$ and let $e=(q^4-1)/(q-1)$. Let $\alpha$ be a primitive element of $ {\rm GF}(q^4)$ and let $I$ be a corresponding Singer $\left(\frac{q^4-1}{q-1}, \frac{q^3-1}{q-1}, q+1\right)$-DS.
  
Let $1 \leq k \leq e-1$ and let $B_k^1, \ldots, B_k^{q+1}$ be the $q+1$ blocks of $\operatorname{Dev}(I)$ (each of size $(q^3-1)/(q-1)$) that contain $\{0,k\}$.  Then $S_k$ is the intersection of any two of these, and 
$$\Delta(C_k,C_0)=\bigcup_{j \in S_k \setminus \{0,k\}} C_j.$$
\end{theorem}
\begin{proof}
Only the claim about $S_k$ requires proof. For $n \geq 2$, since $\operatorname{Dev}(I)$ is a symmetric $(v,k,\lambda)$-BIBD, every pair of DSs intersect in precisely $\lambda=\frac{q^{n-2}-1}{q-1}$ points.  For two such DSs containing $0$ and $k$, the intersection of the corresponding hyperplanes will contain $\operatorname{Span}(1,\alpha^k)$, hence the intersection of the difference sets will contain $S_k$.  When $n=4$, $\lambda=q+1$ and so the intersection of any two DSs containing $\{0,k\}$ will yield precisely the elements of $S_k$.  Alternatively, in ${\rm PG}(3,q)$ a hyperplane is in fact a plane, and any two planes intersect in a unique line.  Thus any two distinct translations of the Singer difference set that contain $\{0,k\}$ intersect in $S_k$.  In Section~\ref{sec:generaln} we discuss how to find appropriate translations.
\end{proof}

\begin{corollary}
Consider the finite field $ {\rm GF}(q^4)$. Let $e=(q^4-1)/(q-1)$ and $f=q-1$. Let $\alpha$ be a primitive element of $ {\rm GF}(q^4)$; let $I$ be a corresponding Singer difference set in $\mathbb{Z}_e$ and let $I_k$ be the intersection of any two  blocks in $\operatorname{Dev}(I)$ containing $\{0,k\}$.  Then
\begin{itemize}
\item[(i)] $(0,0)_e=f-1$
\item[(ii)] $(0,i)_e=(i,0)_e=(i,i)_e=0$ for all $1 \leq i \leq e-1$
\item[(iii)] For $1 \leq j \leq e-1$ and $i \not\in \{0, j\}$, $(i,j)_e$ equals $1$ if $i \in I_j \setminus \{0,j\}$, and equals $0$ otherwise.
\end{itemize}
\end{corollary}

\begin{example}
Consider $ {\rm GF}(625)$; we will obtain the cyclotomic numbers of order $(5^4-1)/(5-1)=156$. No general results for $e=156$ exist in the literature.  Take $q=5$ and $n=4$ with $e=156$ and $f=4$, and take the primitive polynomial $x^4+4x^2+4x+2$. From Construction 18.29 of \cite{Handbook}, a corresponding Singer difference set is $$I=\{0, 1, 2, 4, 13, 20, 23, 24, 29, 31, 34, 38, 41, 44, 46, 58, 72, 73, 77, 88,$$
$$ 89, 95, 97, 98, 111, 120, 124, 139, 144, 150,  152\}.$$
Consider $k=1$: here $I \cap (-1+I)=S_1=\{0,1,23,72,88,97\}$.  Hence $(23,1)_{156}=(72,1)_{156}=(88,1)_{156}=(97,1)_{156}=1$ and all other $(i,1)_{156}=0$. For $k=2$, the intersection of $I$ and $-2+I$ is the set $S_2=\{0,2,29,44,95,150\}$, hence $(29,2)_{156}=(44,2)_{156}=(95,2)_{156}=(150,2)_{156}=1$ and all other $(i,2)_{156}=0$.  The other cyclotomic numbers may be obtained similarly; in Section~\ref{sec:generaln} we discuss how to find suitable shifts of $I$ for other values of $k$.  Lemma \ref{Storer} can be applied to simplify computation.  For example, we may apply $(i,j)_{156}=(5i,5j)_{156}$ to the $k=1$ case to obtain $(115,5)_{156}=(48,5)_{156}=(128,5)_{156}=(17,5)_{156}=1$, as an alternative to calculating $S_5$ directly via intersection.
\end{example}

\subsection{Cyclotomic numbers of order $(q^n-1)/(q-1)$ for general $n$}\label{sec:generaln}

For $n$ greater than $4$, more care is required to obtain $S_k$ by intersection, since an arbitrary collection of $n-2$ hyperplanes may intersect in a space of dimension greater than 1.  

We present a procedure which yields $S_k$ ($1 \leq k \leq e-1)$ for any $n \geq 3$ via the intersection of  $n-2$ appropriately-chosen Singer difference sets, when $\alpha^k$ is not in a subfield of $ {\rm GF}(q^n)$. 

\begin{lemma}\label{alpha_shift} 
Let $\alpha$ be a primitive element of $ {\rm GF}(q^n)$, and suppose that $k$ with $1\leq k\leq e$ is such that $\alpha^k$ does not lie in a proper subfield of ${\rm GF}(q^n)$. 
Let $H_0$ be the unique hyperplane of ${\rm PG}(n-1,q)$ spanned by $\{\alpha^0,\alpha^k,\alpha^{2k}, \ldots \alpha^{(n-2)k}\}$ and for $1\leq i\leq n-1$ let $H_i= \alpha^{-ik}H_0$. Then for $r<n-1$,
$\cap_{i=0}^r H_i$ is the subspace of ${\rm PG}(n-1,q)$ of dimension $n-r-2$ spanned by the set of points $\{\alpha^0,\alpha^1,\ldots,\alpha^{n-r-2}\}$.
\end{lemma}
\begin{proof} As $\alpha^k$ does not lie in a proper subfield of ${\rm GF}(q^n)$, its minimal polynomial has degree $n$, and so the elements $\alpha^0,\alpha^k,\alpha^{2k}, \ldots \alpha^{(n-2)k}$ are linearly independent, and hence span a unique hyperplane $H_0$ of ${\rm PG}(n-1,q)$.  As $H_i$ with $1 \leq r \leq n-1$ can be obtained from $H_0$ by repeated multiplication by $\alpha$, they are also hyperplanes.

Next we prove the intersection property.  When $r=1$, the result is immediate.  Now suppose the result holds for $r=s$, for some $1 \leq s<n$, i.e. that $\cap_{i=0}^s H_i$ is the subspace of dimension $n-s-1$ given by $\operatorname{Span} \{\alpha^0,\alpha^k,\ldots,\alpha^{(n-s-2)k}\}$.  As $H_{s+1}$ is a hyperplane, we have that $\cap_{i=0}^{s+1} H_i$ has dimension either $n-s-1$ or $n-s-2$.  If $\operatorname{dim}(\cap_{i=0}^{s+1} H_i)=n-s-1$, then $\cap_{i=0}^{s+1} H_i$ is precisely $\operatorname{Span} \{\alpha^0,\alpha^k,\ldots,\alpha^{(n-s-2)k}\}$, which implies that $H_{s+1}$ contains $\{\alpha^0,\alpha^k, \ldots, \alpha^{(n-s-2)k}\}$.  By construction, it also contains $\alpha^{-k}, \ldots, \alpha^{-sk}$ and thus $H_s \subseteq H_{s+1}$. Since $|H_s|=|H_{s+1}|$ we have $H_s=H_{s+1}$. But this is contradiction, since $H_s=\alpha^k H_{s+1}$ and multiplication by $\alpha^k$ does not fix any hyperplane when $k<e$.
\end{proof}

\begin{theorem}\label{hyperplanes}
Let $n \geq 3$. Consider the finite field $ {\rm GF}(q^n)$ and let $e=(q^n-1)/(q-1)$.  Let $\alpha$ be a primitive element of $ {\rm GF}(q^n)$.  For $1 \leq k \leq e-1$ such that $\alpha^k$ does not lie in a subfield of $ {\rm GF}(q^n)$, the following process yields $S_k$:
\begin{itemize}
\item[(i)] Let $I$ be any Singer difference set corresponding to $\alpha$.
\item[(ii)] Form the intersection $L$ of the following $n-2$ translates of $I$:
$$L= \bigcap _{i=0}^{n-3} (-ik+I).$$
\item[(iii)] Define the element $x \in L \cap (-k+L)$.
\item[(iv)] Then $S_k=-x+L$.
\end{itemize}
In particular, this accounts for all $1 \leq k \leq e-1$ when $n$ is prime.
\end{theorem}
\begin{proof}
First consider $k=1$; the elements $\{\alpha^0,\alpha,\ldots, \alpha^{n-1}\}$ span a hyperplane $H_0$ and the corresponding Singer DS $I$ contains $\{0,1,\ldots,n-1\}$.  By Lemma \ref{alpha_shift}, the intersection of $H_0$ with the hyperplanes $\{H_1, \ldots, H_{n-3}\}$ is a subspace of dimension $n-(n-3)-1=2$ given by $\operatorname{Span}(1,\alpha)$, and correspondingly the intersection of $I$ with the sets $-1+H, \ldots, -(n-3)+H$ is $S_1$.

Now consider an arbitrary $1 \leq k \leq e-1$ such that $\alpha^k$ is not in a proper subfield.  Here $\alpha^k$ has minimal polynomial of degree $n$, the elements $\{\alpha^0, \alpha^k, \ldots, \alpha^{(n-1)k}\}$ span a hyperplane $H$, and the corresponding Singer difference set $D$ contains $\{0,k,2k,\ldots, (n-1)k\}$.  As above, intersecting $D$ with its translates $\{-k+D, -2k+D, \ldots, -(n-3)k+D\}$ yields the line through $0$ and $k$, namely $S_k$.  However, it is in fact sufficient to take any hyperplane $H'$ of $ {\rm GF}(q^n)$ and the translates of $H'$ by $k$.  These are related to $H$ and its translates by $k$ via a linear transformation, hence they have the same intersection properties (in particular they intersect in a line $L$).  $L$ is obtained from the line through $0$ and $k$ by this linear transformation (multiplication by a power of $\alpha$), and so also possesses a unique internal difference of $k$.  Let $x$ be the unique element in $L \cap (-k+L)$; then $-x+L$ is the translate of $L$ containing $0$ and $k$, namely $S_k$.
\end{proof}

The following result yields $S_k$ in the case when $\alpha^k$ lies in a subfield ${\rm GF}(q^d)$ of ${\rm GF}(q^n)$. In this case $k$ is a multiple of $(q^n-1)/(q^d-1)$.

\begin{theorem}\label{k_in_subfield}
Let $n \geq 3$.  Consider $ {\rm GF}(q^n)$ and let $\alpha$ be a primitive element of $ {\rm GF}(q^n)$.  Let $d\mid n$ ($d>1$).  Let $e=(q^n-1)/(q-1)$, $e'=(q^d-1)/(q-1)$ and $g=(q^n-1)/(q^d-1)=e/e'$.  

Let $\gamma$ be the primitive element of $ {\rm GF}(q^d)$ given by $\gamma=\alpha^g$.
Take the Singer difference set $I$ in $\mathbb{Z}_e$ corresponding to $\alpha$ and the Singer difference set $I'$ in $\mathbb{Z}_{e'}$ corresponding to $\gamma$.  For $0<a<e-1$, let $S_a=\{j: \alpha^j \in \operatorname{Span}(1,\alpha^a), 0 \leq j \leq e-1\} \subseteq \mathbb{Z}_e$ and for $0<b<e'-1$, let $S'_b=\{j: \gamma^j \in \operatorname{Span}(1,\gamma^b), 0 \leq j \leq e'-1\} \subseteq \mathbb{Z}_{e'}$ (where both spans are over $ {\rm GF}(q)$) as in Definition \ref{S_k}.

Then for $k=xg$ $(1 \leq x \leq q^d-2)$ $$S_k=g S'_x.$$
\end{theorem}
\begin{proof}
Let $0<k<e-1$ be such that $\alpha^k \in {\rm GF}(q^d)$.  Then $g\mid k$, i.e. $k=xg$ for some $1 \leq x \leq q^d-2$. To determine $S_{xg}$, the indices of the span over $ {\rm GF}(q)$ of $\{1,\alpha^{xg}\}$, we determine the span over $ {\rm GF}(q)$ of $\{1,\gamma^{x}\}$. Consider the cyclotomic classes of order $e'$ of $ {\rm GF}(q^d)$, i.e. $C_i^{e'}$ ($0 \leq i \leq e'-1$).  By Theorem \ref{thm:main}, $\operatorname{Span}\{1,\gamma^{x}\}$ is the union of the cyclotomic classes of order $e'$ indexed by $S'_{x}$, obtainable using the Singer difference set $I'$.  Since $\gamma=\alpha^g$, multiplying $S'_x$ by $g$ yields $S_{xg}$. 
\end{proof}

We note that in order to apply Theorem \ref{k_in_subfield} we need a difference set arising from the minimal polynomial of $\gamma$.  The Conway polynomials \cite{Nic} or their recent alternative \cite{Lue} are a natural choice for the minimal polynomials of $\alpha$ and $\gamma$ for this purpose, as by construction these are designed to be compatible in this way.  Tables of these polynomials are available on Frank L\"{u}beck's website \url{http://www.math.rwth-aachen.de/~Frank.Luebeck/data/ConwayPol/index.html}.

\begin{theorem}\label{n>4}
Let $n \geq 3$. Consider the finite field $ {\rm GF}(q^n)$. Let $e=(q^n-1)/(q-1)$ and $f=q-1$. Let $\alpha$ be a primitive element of $ {\rm GF}(q^n)$; let $I$ be a corresponding Singer difference set in $\mathbb{Z}_e$. Let $1 \leq k \leq e-1$.  Then $$\Delta(C_k,C_0)=\bigcup_{j \in S_k \setminus \{0,k\}} C_j$$
where $S_k$ is characterized in Theorems \ref{hyperplanes} and \ref{k_in_subfield}, namely
\begin{itemize}
\item[(i)] if $\alpha^k$ does not lie in a subfield of $ {\rm GF}(q^n)$ then 
$$S_k=x+ (\bigcap _{i=0}^{n-3} (-ik+I))$$
where $x \in (\cap _{i=0}^{n-3} (-ik+I)) \cap (-k + \cap _{i=0}^{n-3} (-ik+I))$;
\item[(ii)] if $\alpha^k=\alpha^{y (q^n-1)/(q^d-1)}$ lies in $ {\rm GF}(q^d)$ ($d>1$, $d\mid n$) then $$S_k= \left(\frac{q^n-1}{q^d-1}\right) S'_y$$
 where $S'_y=\{j: \gamma^j \in  \operatorname{Span}(1,\gamma^y), 0 \leq j \leq e'-1\}$.
\end{itemize}
Moreover
\begin{itemize}
\item[(a)] $(0,0)_e=f-1$
\item[(b)] $(0,i)_e=(i,0)_e=(i,i)_e=0$ for all $1 \leq i\leq e-1$
\item[(c)] For $1 \leq j \leq e-1$ and $i \not\in \{0, j\}$, $(i,j)_e$ equals $1$ if $i \in S_j \setminus \{0,j\}$, and equals $0$ otherwise.
\end{itemize}
\end{theorem}

\begin{example}
We consider $ {\rm GF}(729)$ as a degree $6$ extension over $ {\rm GF}(3)$; here $e=(3^6-1)/(3-1)=364$.  We illustrate the use of Theorem \ref{n>4} to determine $S_k$ for $1 \leq k \leq 363$, and hence the cyclotomic numbers of order $364$; there is no result in the literature for the case $e=364$.  Take primitive polynomial $x^6+2x^4+x^3+x^2+x+2$ for $ {\rm GF}(729)$ over $ {\rm GF}(3)$ and primitive polynomial $x^3+x^2+2x+1$ for $ {\rm GF}(27)$ over $ {\rm GF}(3)$ ($\alpha$ is a root of the former and $\gamma=\alpha^{28}$ is a root of the latter).

For $k$ such that $\alpha^k$ is not in a subfield, we determine $S_k$ as described in Theorem \ref{hyperplanes}: we take the corresponding Singer difference set $I$ (of size $\frac{3^5-1}{3-1}=121$) and intersect it with its translates $-k+I, -2k+I, -3k+I$ to obtain $L$, then shift by $L \cap (-k+L)$.  For example, for $k=1$ this yields $S_1=\{ 0, 1, 27, 322 \}$ and for $k=2$ we obtain $S_2=\{0,2, 90,349\}$.  Hence $(27,1)_{364}=(322,1)_{364}=1$ and $(x,1)_{364}=0$ for all other $1 \leq x \leq 363$; similarly $(90,2)_{364}=(349,2)_{364}=1$ and $(x,2)_{364}=0$ for all other $1 \leq x \leq 363$.  

Now, consider $k$ such that $\alpha^k$ lies in the subfield $ {\rm GF}(27)=\langle \gamma \rangle$.  In the notation of Theorem \ref{k_in_subfield}, $e'=13$ and $g=28$, so any such $k$ has the form $28x$ for some $1 \leq x \leq 26$.  A suitable Singer difference set is $I'=\{0,1,8,10\}$; since $ {\rm GF}(27)$ is a degree $3$ extension over $ {\rm GF}(3)$, each $S'_x$ will be a translate of $I'$. For $k=28$, $S'_1=\{0,1,8,10\}$ and so $S_{28}=\{0,28,224,280\}$. Consider $k=56=28.2$; here $S'_2=\{0,2,5,6\}$, and so $S_{56}=28 \{0,2,5,6\}=\{0,56,140,168\}$.   Hence $(224,28)_{364}=(280,28)_{364}=1$ and $(x,28)_{364}=0$ for all other $1 \leq x \leq 363$; similarly $(140,56)_{364}=(168,56)_{364}=1$ and $(x,56)_{364}=0$ for all other $1 \leq x \leq 363$.
\end{example}

\section{Cyclotomic numbers of order $\epsilon\mid (q^n-1)/(q-1)$}

We next show how to use a partitioning construction to obtain results for cyclotomic numbers of order $\epsilon$ dividing  $(q^n-1)/(q-1)$.  The following lemma is from \cite{HucJoh}.
\begin{lemma}\label{union}
Let $q^n=ef+1$ and let $\epsilon\mid e$.  For $0 \leq j \leq \epsilon-1$, 
$$C_j^{\epsilon}=\cup_{i=0}^{\frac{e}{\epsilon}-1} C_{i \epsilon +j}^e.$$
In particular $C_0^{\epsilon}=C_0^e \cup C_{\epsilon}^e \cup \cdots \cup C_{\epsilon (\frac{e}{\epsilon}-1)}^e$.
\end{lemma}

Combining this with the previous section, we have the following.
\begin{theorem}\label{GeneralCycTheorem}
Consider the finite field $ {\rm GF}(q^n)$, where $n \geq 2$.  Let $e=(q^n-1)/(q-1)$ and let $\alpha$ be a primitive element of $ {\rm GF}(q^n)$. Let $q^n=ef+1$ and let $\epsilon\mid e$. Then
\begin{itemize}
\item[(i)] $$(0,0)_{\epsilon}=(f-1)+\sum_{s=1}^{\frac{e}{\epsilon}-1} \sum_{\substack{r \equiv 0 \bmod \epsilon,\\ 1 \leq r \leq e-1}} (r,\epsilon s)_e$$
\item[(ii)] For all $0<i<e$, 
$$(i,0)_{\epsilon}=(0,i)_{\epsilon}=(e-i,e-i)_{\epsilon}=\sum_{s=1}^{\frac{e}{\epsilon}-1} \sum_{\substack{r \equiv i \bmod \epsilon,\\ 1 \leq r \leq e-1}} (r,\epsilon s)_e$$
\item[(iii)] For $0 < i \neq j <e$, 
$$(i,j)_{\epsilon}=\sum_{s=0}^{\frac{e}{\epsilon}-1} \sum_{\substack{r \equiv i \bmod \epsilon \\ 1 \leq r \leq e-1}} (r,\epsilon s+ j)_e$$
\end{itemize}
\end{theorem}
\begin{proof}
For (i) and (ii), we determine the cyclotomic classes of order $\epsilon$ in $\Delta(C_0^{\epsilon})$.  By Lemma \ref{DM}, $\Delta(C_0^{\epsilon})=\cup_{s=1}^{\frac{ef}{\epsilon}-1} (\alpha^{ \epsilon s}-1)C_0^{\epsilon}=\cup_{i=0}^{\epsilon-1} (i,0)_{\epsilon} C_i^{\epsilon}$.
By Lemma \ref{union}, $\Delta(C_0^{\epsilon})= (\cup_{s=0}^{\frac{e}{\epsilon}-1} \Delta(C_{s \epsilon}^e)) \cup ( \cup_{t=1}^{\frac{e}{\epsilon}-1} \cup_{s=0}^{\frac{e}{\epsilon}-1}  \Delta(C_{(s+t)\epsilon}^e, C_{s\epsilon}^e))$.  
Each $C_i^e$ in $\Delta(C_0^e)$ corresponds to an occurrence of $C_{i \bmod \epsilon}^{\epsilon}$ in $\cup_{s=0}^{\frac{e}{\epsilon}-1}\Delta(C_{s\epsilon}^e)$, and each $C_i^e$ in the multiset $\Delta(C_{t \epsilon}^e,C_0^e)$ ($1 \leq t \leq (e/\epsilon)-1$) corresponds to an occurrence of $C_{i \bmod \epsilon}^{\epsilon}$ in $\cup_{s=0}^{\frac{e}{\epsilon}-1} \alpha^{s \epsilon} \Delta(C_{t \epsilon}^e,C_0^e)=\cup_{s=0}^{\frac{e}{\epsilon}-1}\Delta(C_{(s+t)\epsilon}^e,C_{s \epsilon}^e)$. Hence the $ef/\epsilon-1$ classes of order $e$ in $\Delta(C_0^e)$ and $\Delta(C_{t \epsilon}^e,C_0^e)$ ($0 <t<\frac{e}{\epsilon}$) determine the classes of order $\epsilon$ in $\Delta(C_0^{\epsilon})$.
For (i), since $C_0^e$ is a subfield, $\Delta(C_0^e)$ comprises $f-1$ copies of $C_0^e$ and no copies of any other class, yielding $f-1$ copies of $C_0^{\epsilon}$.  The other occurrences of $C_0^{\epsilon}$ in $\Delta(C_0^{\epsilon})$ will correspond to occurrences of $C_r^e$ in $\Delta(C_{s \epsilon}^e,C_0^e)$ ($1 \leq s \leq \frac{e}{\epsilon}-1$) where $0<r<e$ and $\epsilon\mid r$; these are counted by $(r,\epsilon s)_e$.  Similarly, for (ii), the occurrences of $C_r^e$ in $\Delta(C_0^{\epsilon})$ with $0<r<e$ and $r \equiv i \bmod \epsilon$ account for the cyclotomic numbers $(i,0)_{\epsilon}$ (and hence $(0,i)_{\epsilon}=(-i,-i)_{\epsilon}$ by Lemma \ref{Storer}, as $p=2$ or $f\mid \frac{q^n-1}{\epsilon}$ and $f$ is even).  

For (iii), for $0<j<\epsilon$, consider $\Delta(C_j^{\epsilon}, C_0^{\epsilon})=\cup_{s=0}^{\frac{ef}{\epsilon}} (\alpha^{ \epsilon s+j}-1) C_0^{\epsilon}  =\cup_{i=0}^{\epsilon-1} (i,j)_{\epsilon} C_i^{\epsilon}=\cup_{s=0}^{\frac{e}{\epsilon}-1} \cup_{t=0}^{\frac{e}{\epsilon}-1} \Delta(C_{s\epsilon+j}^e, C_{t\epsilon}^e)$.  Each $C_i^e$ in the multiset $\Delta(C_{t \epsilon+j}^e,C_0^e)$ corresponds to an occurrence of $C_{i \bmod \epsilon}^{\epsilon}$ in $\cup_{s=0}^{\frac{e}{\epsilon}-1} \Delta(C_{(s+t) \epsilon+j}^e, C_{s \epsilon}^e)$. These account for all $ef/\epsilon$ classes of order $\epsilon$ occurring in $\Delta(C_j^{\epsilon})$ and are counted by $(r, \epsilon s+j)_e$ where $0<r<\epsilon$ and $r \equiv i \bmod \epsilon$.
\end{proof}

When $n=2$, we obtain an explicit expression for the uniform cyclotomy case which coincides with that of \cite{BauMilWar}, showing that the cyclotomic numbers of order $\epsilon\mid q+1$ in $ {\rm GF}(q^2)$ are uniform. It also coincides with the standard results \cite{Sto} for $\epsilon=2$.

\begin{corollary}\label{UniformCor}
Let $q$ be a power of a prime $p$ and consider $ {\rm GF}(q^2)$. Let $q^2=ef+1$ with $e=q+1$ and $f=q-1$.  If  $\epsilon \geq 2$ is a divisor of $e$, then
\begin{itemize}
\item[(i)] $(0,0)_{\epsilon}=(f-1)+(\frac{e}{\epsilon}-1)(\frac{e}{\epsilon}-2)=(\frac{e}{\epsilon})^2+(\epsilon-3) \frac{e}{\epsilon}-1$
\item[(ii)] $(i,0)_{\epsilon}=(0,i)_{\epsilon}=(\epsilon-i,\epsilon-i)_{\epsilon}=\frac{e}{\epsilon}(\frac{e}{\epsilon}-1)$
\item[(iii)] $(i,j)_{\epsilon}=(\frac{e}{\epsilon})^2$.
\end{itemize}
\end{corollary}
\begin{proof}
In the case $n=2$, each $\Delta(C_{s \epsilon}^e,C_0^e)$ ($1 \leq s \leq \frac{e}{\epsilon}-1)$ contains one copy each of the classes $C_i^e$ where $i \in \{1,\ldots,e-1\} \setminus \{s \epsilon\}$.  Each of the classes $C_i^e$ where $\epsilon\mid i$ yields a copy of $C_0^{\epsilon}$ in $\Delta(C_0^{\epsilon})$.  In the interval $[1,e-1]=\{1,2,\ldots (\frac{e}{\epsilon}-1)\epsilon+(\epsilon-1)\}$ there are $\frac{e}{\epsilon}-1$ multiples of $\epsilon$, so each of the $(\frac{e}{\epsilon}-1)$ multisets $\Delta(C_{s \epsilon}^e,C_0^e)$ contributes $(\frac{e}{\epsilon}-2)$ copies of $C_0^{\epsilon}$ to the count for $(0,0)_{\epsilon}$ (the other part of this is from $\Delta(C_0^e)$). For (ii), when $0<j<\epsilon$, there are $\frac{e}{\epsilon}$ elements in $[1,e-1]$ congruent to $j \bmod \epsilon$, and each occurs once per  $\Delta(C_{s \epsilon}^e, C_0^e)$ as $1 \leq s \leq \frac{e}{\epsilon}-1$.  Finally for (iii), for any $1 \leq r \leq e-1$ with $r \equiv i \bmod \epsilon$, if $1 \leq j \leq e-1$ with $i \neq j$ then all $(r,s \epsilon+j)$ equal $1$ (since $r \not\in \{0,s\epsilon+j\})$. Hence there are $\frac{e}{\epsilon}$ occurrences of each $C_r^{\epsilon}$ in each  $\Delta(C_{s \epsilon}^e,C_0^e)$.
\end{proof}

\begin{example} When $q=25$, $e=6$, $f=4$ and $\epsilon=3$, Theorem \ref{BMWTheorem} applies since $-1 \equiv 5^3 \mod 3$. In Theorem \ref{BMWTheorem} we take $r=-5$; then $\eta=-2$ and so $(0,0)_3=2^2-1=3$, $(i,0)_3=4-2=2$ ($i \neq 0$) and $(i,j)_3=(-2)^2=4$ ($0< i \neq j <3$). Since $3\mid 6$, we may alternatively apply Corollary \ref{UniformCor} from which $e/\epsilon=2$, $(0,0)_3=3+0=3$, $(i,0)_3=2 \times 1=2$ ($i \neq 0$) and $(i,j)_3=2^2=4$ ($0<i \neq j<3$).  When $\epsilon=2<3$, we cannot apply Theorem \ref{BMWTheorem}, but Corollary \ref{UniformCor} (with $(q-1)/\epsilon=12$) yields the expected results $(0,0)_2=(12-2)/2=5$ and $(1,0)_2=12/2=6$.
\end{example}

As mentioned at the outset, Theorem \ref{GeneralCycTheorem} also applies to the uniform cyclotomy situation in ${\rm GF}(p^{2s})$ where we consider the cyclotomic numbers of order $\epsilon\mid p^t+1$ where $t$ is a proper divisor of $s$. We illustrate with an example.

\begin{example}
Consider ${\rm GF}(81)$.  Here $q=p^{2s}$ where $p=3$ and $s=2$. Corollary \ref{UniformCor} shows that the cyclotomic numbers of order dividing $(p^4-1)/(p^2-1)=p^2+1=10$ are uniform.  Now consider the situation where $\epsilon=p+1=4$; by Theorem \ref{BMWTheorem}, $\epsilon$ satisfies the criteria to be uniform, but $4 \nmid p^2+1$.  Observe that $p+1\mid (p^4-1)/(p-1)=(p^2+1)(p+1)$.  So we may use Theorem \ref{GeneralCycTheorem} with $e=(3^4-1)/(3-1)=40$, $f=2$ and $\epsilon=4$ to obtain the cyclotomic numbers of order $4$, proceeding via the cyclotomic numbers of order $40$.  Obtaining the cyclotomic numbers of order $40$ is itself of interest, since no general results for $e=40$ exist in the literature.

Using the primitive polynomial $x^4+2x^3+2$, we obtain the difference set
$$I=\{0, 1, 2, 8, 16, 18, 23, 25, 28, 29, 34, 37, 38 \}.$$
This is a degree $4$ extension so the sets $S_k$ ($0<k<39$) are obtained by intersecting two translates of this difference set. For example $S_1=I \cap (-1+I)$, where $-1+I=\{0,1,7,15,17,22,24,27,28,33,36,37,39\}$.  The complete list of sets $S_k$ are as follows:
$$ \{0, 1, 28, 37\}, \{0, 2,18,25\}, \{0, 3, 4, 31\}, \{0, 5,11, 19\},\{0, 6, 14,35\}, $$
$$\{0, 7, 22, 24\}, \{0, 8, 29, 34\}, \{0, 9, 12, 13\},\{0, 10, 20, 30\}, \{0, 15, 17, 33\},$$
$$\{0, 16, 23, 38\}, \{0, 21, 26, 32\}, \{0, 27, 36, 39\}.$$
Apply Theorem \ref{GeneralCycTheorem} (i) to see that, since no $S_{4s}$ ($0<s<10$) contains an element $4t$ with $0<s \neq t<10$, we have $(0,0)_4=f-1=1$.  To apply Theorem \ref{GeneralCycTheorem} (ii), note that for each $r \in \{1,2,3\}$, precisely $6$ of the cyclotomic numbers $(r,4s)_{40}$ ($0<s<10$) are nonzero (specifically, have value $1$).  For example for $r=1$, $(29,8)_{40}=(9,12)_{40}=(13,12)_{40}=(1,28)_{40}=(37,28)_{40}=(21,32)_{40}=1$, while for $r=2$, $(34,8)_{40}=(38,16)_{40}=(10,20)_{40}=(30,20)_{40}=(22,24)_{40}=(26,32)_{40}=1$.  This indicates that $(0,i)_{4}=(i,0)_{4}=(i,i)_{4}=6$ for all $0<i<4$.  Finally, using Theorem \ref{GeneralCycTheorem} (iii), we can verify that $(i,j)_4=4$ for all $0<i \neq j <4$.  For example for $(2,1)_4$ we use the fact that $(32,21)_{40}=(2,25)_{40}=(28,25)_{40}=(34,29)_{40}=1$.  Hence the cyclotomic numbers of order $p+1=4$ are uniform. Theorem \ref{BMWTheorem} yields the same result upon taking $r=9$ and $\eta=2$.
\end{example}

Since the previous section gives a process for obtaining the component cyclotomic numbers of order $e$, Theorem \ref{GeneralCycTheorem} gives an explicit method to find the cyclotomic numbers of any order dividing $(q^n-1)/(q-1)$ where $q$ is a prime power and $n \in \mathbb{N}$ with $n \geq 2$. Lemma \ref{Storer} often significantly reduces the calculation required.

\begin{example}
We demonstrate the use of Theorem \ref{GeneralCycTheorem} to directly obtain the cyclotomic numbers of order $13$ in $ {\rm GF}(729)$. As mentioned previously, no direct result for cyclotomic numbers of order $13$ is known in general.   Set $q=9$ and $n=3$ in Theorem \ref{GeneralCycTheorem}, i.e. consider $ {\rm GF}(729)$ as a degree $3$ extension over $ {\rm GF}(9)$.  We use the primitive poynomial $x^3 + x + w$ where $w$ is given by $y^2 + 2y + 2$.   Here $e=91$, $f=8$ and we take $\epsilon=13$.  We proceed via the cyclotomic numbers of order $91$ in $ {\rm GF}(729)$; we use the Singer difference set $I=\{0,1,3,9,27,49,56,61,77,81\} \subseteq \mathbb{Z}_{91}$ and its $0$-containing translates $\{-i+I: i \in I\}$:
$$\{10,11,13,19,37,59,66,71,87,0\},\{14,15,17,23,41,63,70,75,0,4\},$$
$$\{30,31,33,39,57,79,86,0,16,20\}, \{35,36,38,44,62,84,0,5,21,25\},$$
$$\{42,43,45,51,69,0,7,12,28,32\}, \{64,65,67,73,0,22,29,34,50,54\},$$
$$ \{82,83,85,0,18,40,47,52,68,72\}, \{88,89,0,6,24,46,53,58,74,78\},$$
$$ \{90,0,2,8,26,48,55,60,76,80\}.$$
Here $(0,0)_{13}=f-1=7$, since for each $13s$ ($1 \leq s \leq 6$), the unique translate of $I$ containing $0$ and $13s$ contains no other $13t$ ($1 \leq t \leq 6$), and hence $\sum_{s=1}^{6} \sum_{\substack{r \equiv 0 \bmod 13,\\ 1 \leq r \leq 90}} (r,13s)_{91}=0$.  By similarly evaluating the cyclotomic numbers of the form $(r,13s)_{91}$ where $r \equiv i \mod 13$ ($0<i<91$) using the difference sets listed above, the cyclotomic numbers of the form $(x,0)_{13}$ ($x \neq 0$) are found to be $(1,0)_{13}=4$, $(2,0)_{13}=4$, $(3,0)_{13}=4$, $(4,0)_{13}=2$, $(5,0)_{13}=4$, $(6,0)_{13}=4$, $(7,0)_{13}=6$, $(8,0)_{13}=6$, $(9,0)_{13}=4$, $(10,0)_{13}=2$, $(11,0)_{13}=4$.  The cyclotomic numbers of the form $(x,1)_{13}$ are given by $(0,1)_{13}=4$, $(1,1)_{13}=2$, $(2,1)_{13}=2$, $(3,1)_{13}=6$, $(4,1)_{13}=6$, $(5,1)_{13}=5$, $(6,1)_{13}=4$, $(7,1)_{13}=6$, $(8,1)_{13}=2$, $(9,1)_{13}=6$, $(10,1)_{13}=6$, $(11,1)_{13}=5$ and $(12,1)_{13}=2$. The same process yields all of the other cyclotomic numbers of order $13$.
\end{example}

\section{Applications to difference structures}

In this section, we briefly discuss applications of our results in the setting of difference structures.  There has been much interest in the literature in using cyclotomic classes and their unions to form various types of difference families and difference sets, beginning with the work of Wilson \cite{Wil} in the 1970s.

The following structure was introduced in \cite{PatSti}, motivated by an application in information security. For a subset $S$ of a group $G$, we denote by $\lambda S$ the multiset comprising $\lambda$ copies of $S$.
\begin{definition}
Let $G$ be an abelian group of order $n$, written additively.  A collection of $m$ disjoint $k$-subsets $\mathcal{D}^{\prime} = \{D_1,\ldots,D_m\}$ of $G$ forms an $(n,m,k,\lambda)$-Strong External Difference Family (SEDF) of $G$ if the following multiset equation holds for each $D_i \in \mathcal{D}^{\prime}$:
\begin{equation*}
\bigcup_{j: j \neq i} \Delta(D_i,D_j) = \lambda(G \setminus \{0\}).
\end{equation*}
\end{definition}

SEDFs have attracted particular interest since only one set of parameters with $m>2$ is known.  Constructions achieving these parameters were found independently by two sets of authors in \cite{JedLi} and \cite{WenYanFuFen}.  The construction in \cite{WenYanFuFen} uses the cyclotomic classes of order $11$ in $ {\rm GF}(243)$.

\begin{example}
Cyclotomic numbers of order $11$ in $ {\rm GF}(243)$ were required in \cite{WenYanFuFen} to establish existence of the $(243,11,22,20)$-SEDF.  While a treatment of the $e=11$ cyclotomic numbers using Jacobi sums is given in \cite{LeoWil}, it is very cumbersome to use. Consequently the authors presented four pages of direct calculations in the field (given in \cite{WenArxiv}) to establish that the values of $(0,0)_{11}=1$ and $(i,0)_{11}=2$ ($1 \leq i \leq 10$).  We show how our approach offers a shorter, alternative method of evaluation.

Take $q=3$, $n=5$, $e=121$, $f=2$ and $\epsilon=11$ in Theorem \ref{n>4}, and use the Singer difference set $$I=\{1,3,4,7,9,11,12,13,21,25,27,33,34,36,39,44,55,63,64,67,68,70,71,75,80,81, $$
$$ 82,83,85,89,92,99,102,103,104,108,109,115,117,119\}$$ given in \cite{Handbook} (or generated by Construction \ref{Handbook}).  By Theorem \ref{hyperplanes}, we obtain
$$L= I \cap (-11+I) \cap (-22+I)=\{33,70,81,108\};$$ 
here 
$$L \cap (-11+L)=\{70\}$$ and translation of $L$ by  $-70$ yields $S_{11}=\{0,11,38,84\}$. Similarly, upon intersection of $I, -22+I$ and  $-44+I$ and appropriate translation,  $S_{22}=\{0,1,22,52\}$.  Hence $(38,11)_{121}=(84,11)_{121}=1=(1,22)_{121}=(52,22)_{121}$ and all other $(x,11)_{121}$ and $(x,22)_{121}$ equal $0$ ($1 \leq x \leq 120$). 

Repeated application of $(i,j)_e=(3i,3j)_e$ from Lemma \ref{Storer} yields all other cyclotomic numbers $(i,j)_e$ with $\epsilon\mid j$ which take nonzero values: $(10,33)_{121}=(114,33)=(53,44)_{121}=(28,44)_{121}=(58,55)_{121}=(90,55)_{121}=(3,66)_{121}=(35,66)_{121}=(9,77)_{121}=(105,77)_{121}=(81,88)_{121}=(98,88)_{121}=(30,99)_{121}=(100,99)_{121}=(27,110)_{121}=(73,110)_{121}=1$. Observe that of these 20 nonzero cyclotomic numbers $(r,s \epsilon)$ ($1 \leq s \leq \frac{e}{\epsilon}-1=10$), none have $\epsilon\mid r$ and for each $1 \leq i \leq 10$ there are precisely 2 values of $r \equiv i \mod {13}$.  Theorem \ref{GeneralCycTheorem} applies to show that $(0,0)_{11}=f-1=1$ and $(i,0)_{11}=2$ for all $1 \leq i \leq 10$.
\end{example}

\section{Conclusion}
The sparsity of cyclotomic number results for specific orders $e$ (indeed, the lack of results for $e > 24$), and the cumbersome nature of obtaining explicit results from the character theory approach for all but small $e$, has been a limiting factor on the practical use of cyclotomy. Until now, uniform cyclotomy has been the only explicit theoretical tool available, but is somewhat limited in its scope.  In this paper, we have presented a generalisation of uniform cyclotomy which is expressed in purely combinatorial terms,  and applies to all cyclotomic numbers of order $e \mid (q^n-1)/(q-1)$ over ${\rm GF}(q^n)$, for any prime power $q$ and $n \geq 2$.  We have presented a direct method of obtaining such cyclotomic numbers, which requires neither use of character theory nor direct calculation with field elements.  We hope that this will go some way towards demystifying cyclotomy for practitioners, and will facilitate an easier and more effective use of cyclotomy in various applications, as well as leading to further research in this direction.

\subsection*{Acknowledgements}
We thank Jim Davis for helpful discussions.  We thank the anonymous referee for their careful reading and helpful comments.

\end{document}